\documentclass{amsart}
\usepackage{latexsym, amssymb,amsthm,amsmath,verbatim}

\begin{document}

\newtheorem{theorem}{Theorem}
\newtheorem{remark}[theorem]{Remark}
\newtheorem{proposition}[theorem]{Proposition}
\newtheorem{question}[theorem]{Question}
\newtheorem{definition}[theorem]{Definition}
\newtheorem{corollary}[theorem]{Corollary}
\newtheorem{lemma}[theorem]{Lemma}
\newtheorem{claim}[theorem]{Claim}
\newtheorem{example}[theorem]{Example}

\newcommand{\pr}[1]{\left\langle #1 \right\rangle}
\newcommand{\mH}{\mathcal{H}}
\newcommand{\mK}{\mathcal{K}}
\newcommand{\mR}{\mathcal{R}}
\newcommand{\mG}{\mathcal{G}}
\newcommand{\mA}{\mathcal{A}}
\newcommand{\mV}{\mathcal{V}}
\newcommand{\mU}{\mathcal{V}}
\newcommand{\mP}{\mathcal{P}}
\newcommand{\mB}{\mathcal{B}}
\newcommand{\C}{\mathrm{C}}
\newcommand{\mO}{\mathcal{O}}
\newcommand{\mC}{\mathcal{C}}
\newcommand{\D}{\mathrm{D}}
\renewcommand{\O}{\mathrm{O}}
\newcommand{\K}{\mathrm{K}}
\newcommand{\OD}{\mathrm{OD}}
\newcommand{\Do}{\D_\mathrm{o}}
\newcommand{\sone}{\mathsf{S}_1}
\newcommand{\gone}{\mathsf{G}_1}
\newcommand{\sfin}{\mathsf{S}_\mathrm{fin}}
\newcommand{\gfin}{\mathsf{G}_\mathrm{fin}}
\newcommand{\gn}[1]{\mathsf{G}_\mathrm{#1}}
\newcommand{\Em}{\longrightarrow}
\newcommand{\menos}{{\setminus}}
\newcommand{\w}{{\omega}}
\newcommand{\1}{\textsc{Alice}}
\newcommand{\2}{\textsc{Bob}}
\newcommand{\seq}[1]{{\langle {#1} \rangle}}

\title[A definitive improvement...]
{A definitive improvement of  a  game-theoretic
bound and the long tightness game}

\author[L. F. Aurichi]{Leandro F. Aurichi}

\address{Instituto de Ci\^encias Matem\'aticas e de Computa\c
c\~ao,
Universidade de S\~ao Paulo, Caixa Postal 668,
S\~ao Carlos, SP, 13560-970, Brazil}
\email{aurichi@icmc.usp.br}

\author[A. Bella]{Angelo Bella}
\address{ Dipartimento di Matematica e Informatica, University of
Catania,
Citt\`a
  Universitaria, Viale A. Doria 6, 95125 Catania, Italy}
\email{bella@dmi.unict.it}

\begin{abstract}The main goal of the paper is  the full proof of
    a cardinal inequality for a space with points  $G_\delta $,
obtained with the help of a long version of the Menger game. This
result, which  improves a similar one  of Scheepers and Tall, was
already established  by the authors  under the Continuum
Hypothesis. The paper is completed by few remarks on a long
version of the tightness game.
\end{abstract}

\maketitle
\smallskip
\section{Introduction}
As usual, for notation and undefined notions we refer to
\cite{engelking}.  In this paper we consider the long version of
two well-known topological games. In particular,  we study  the
influence of the existence of a winning strategy for the second
player in both games to certain cardinality properties  of the
space.

The main result (Theorem \ref{main}) shows that a cardinality
bound,
obtained by Scheepers and Tall with the
help of  the Rothberger game,  continues to hold   with the much
weaker help of the Menger game. Our generalization  works in the
class of  regular spaces and we will remark that  some
separation axiom is definitely needed for it (see Example
\ref{ex}).

The second part of the paper deals with a long version of the
tightness game.  Although this game is very different from the
Menger game,  the main result here, Theorem \ref{main2}, looks
quite similar to Theorem \ref{main}.
\smallskip
\section{  Long Menger game and cardinality}
After Arhangel'ski\u\i's cardinal inequality: $|X|\le
2^{\omega}$,
for any first countable Lindel\"of $T_2$ space $X$,  a lot of
attention has been paid to the possibility of extending this
theorem
to the whole class of spaces with $G_\delta$ points (see e. g.
\cite{hodel}).  The problem
turned out to be very non-trivial and the first negative
consistent answer was given by Shelah.  Later on, a
simpler construction of a Lindel\"of $T_3$ space with points
$G_\delta$ whose cardinality is bigger than  the continuum was
 proposed by Gorelic \cite{gorelic}.
  Somewhat related to
the Lindel\"of property are the Rothberger and Menger games (see
e.
g. \cite{JMSS}). Indeed,
by working in this direction,
in 2010  Scheepers and Tall
\cite{ScheepersTall}  proved a cardinality bound   for a
topological space with points $G_\delta$  by means of a long
version of the Rothberger
game.  The natural question to extend this result to
the much weaker Menger game was studied in \cite{AB}. There, a
 partial answer
was obtained under the Continuum Hypothesis. The main
purpose of this note is to  provide   the  full  solution to the
question in ZFC.  The proof we present here  uses elementary
submodels and
looks much simpler and direct.
\smallskip

We follow the standard notation for games: we will denote
by $\gone^\kappa(\mA, \mB)$ the game played by players $\1$ and
$\2$ such that,  at each inning $\xi < \kappa$, $\1$
chooses $A_\xi \in \mA$. Then $\2$ chooses $a_\xi \in
A_\xi$. $\2$ wins if $\{a_\xi: \xi < \kappa\} \in \mB$.

We
will denote by $\O$ the family of all open covers for a given
space. Thus, $\gone^\kappa(\O, \O)$ means that at each inning
$\1$ chooses an open cover and $\2$  chooses one
of
its members. $\2$  wins if the collection of the
chosen sets   covers the space.

 According to this notation, $\gone^\omega(\O,\O) =
\gone(\O, \O)$ is the classical  Rothberger game.

 As usual, $\mathfrak c = 2^{\w}$.

The starting point  of our investigation is in the following:
\begin{theorem}[Scheepers-Tall, \cite{ScheepersTall}]
\label{start}
 If $X$ is a space with points $G_\delta $   and $\2$ has
a winning strategy in the game $\gone^{\omega_1}(\O, \O)$, then
$|X|\le 2^{\w}$.
\end{theorem}

To appreciate the strength of  the above result and consequently
of Theorem \ref{main} below, notice that
the example of Gorelic \cite{gorelic} provides a space
$X$ with points $G_\delta $ in which  $\1$  does not have
a
winning strategy in $\gone^{\w_1}(\O, \O)$ and
$|X|>2^{\omega}$ (see \cite{ScheepersTall} for a justification
of this fact).

A very natural question arises on  whether  Scheepers-Tall's
inequality can be improved by replacing ``$\gone$''  with
``$\gfin$'', \emph{i.e.}, the game where $\2$  chooses
finitely
many sets per inning, instead of only one. In other words, we
wonder  whether the long  Menger game can suffice in the above
cardinal inequality.

 We  already obtained a positive partial answer under the
continuum
hypothesis in \cite{AB}. Our goal here is to present a proof of
this statement in ZFC.

  In \cite{AB} it is used the
duality
between $\gfin^{\w_1}(\O, \O)$ and the compact-open game of
length $\w_1$. This duality is true under CH but we do not know
if it is true in general. The proof presented here does not use
any duality.

From now on, let $\sigma$ be a fixed winning strategy for $\2$ in
the game $\gfin^\kappa(\O, \O)$ played on the space $X$. Recall
that  a
strategy for $\2$  in $\gfin^\kappa (\O,\O)$ is a function
$\sigma :\O^{<\kappa }=\bigcup\{{}^{\alpha +1}\O :\alpha <\kappa
\}\rightarrow
[\bigcup\O]^{<\omega}$ and for  any $s\in {}^{\alpha +1}\O$
we
have $\sigma (s)\subset  s(\alpha )$.

We will say that $K \subset X$ is {\bf good} if there is an $s
\in \O^{<\kappa}$ such that $K = \bigcap_{\mC \in
\O}\overline{\bigcup \sigma(s^\smallfrown \mC)}$.

\begin{lemma} \label{good} Every good subset of a regular space
is compact. \end{lemma}
\begin{proof} Let $K=\bigcap_{\mC\in \O}\overline {\bigcup\sigma
(s^\smallfrown \mC)}$ and take a collection $\mU$ of open sets
such that $K\subset \bigcup \mU$.  Fix a neighbourhood assignment
$\mV=\{V_x:x\in X\}$ in such a way that $\overline {V_x}\subset
U_x\in \mU$ if $x\in K$ and $\overline {V_x}\cap K=\emptyset $ if
$x\in X\setminus K$. If $\sigma (s^\smallfrown \mV)=\{V_x:x\in
F\in [X]^{<\omega}\}$, then we clearly have $K\subset \bigcup
\{\overline {V_x}:x\in F\cap K\}\subset \bigcup\{U_x:x\in F\cap
K\}$. \end{proof}

\begin{lemma} \label{intersec}
  Let $X$ be a space. If $K$ is good, \emph{i.e.} there
is an $s \in \O^{<\kappa}$ such that $K = \bigcap_{\mC \in
\O}\overline{\bigcup \sigma(s^\smallfrown \mC)}$, and $K =
\bigcap_{\xi < \lambda} V_\xi$ where each $V_\xi$ is open, then
there is an $\O' \subset \O$ such that $K = \bigcap_{\mC \in
\O'}\overline{\bigcup \sigma(s^\smallfrown \mC)}$ and $|\O'| <
\lambda + \kappa $.
\end{lemma}
\begin{proof}  As we are assuming that $\2$ has a winning
strategy in  $\gfin^{ \kappa }(\O,\O)$, the Lindel\"of degree of
$X$ is at most $ \kappa $. Consequently,   we have $L(X\setminus
K)\le \lambda + \kappa $.  Since the family $\{X\setminus
\overline {\bigcup\sigma(s^\smallfrown \mC)}:\mC\in \O\}$ is an
open cover of $X\setminus  K$,  there exists $\O'\subset \O$ such
that $|\O'|\le \lambda + \kappa $ and $\{X\setminus \overline
{\bigcup\sigma(s^\smallfrown \mC)}:\mC\in \O'\}$ covers
$X\setminus K$. Therefore, $K=\bigcap_{\mC\in
\O'}\overline{\bigcup\sigma(s^\smallfrown \mC)}$ and we are 
done.  \end{proof}

\begin{lemma}
Let $X$ be a space with points $G_\delta$. Then for every compact
subset $K$ there is a sequence $\seq{V_\xi: \xi < 2^\w}$ of open
sets such that $K = \bigcap_{\xi < 2^\w} V_\xi$.
\end{lemma}

\begin{proof}
  First note that each compact $K \subset X$  satisfies $|K|
\leq  2^{\w}$. This is a consequence of a theorem of
Gryzlov \cite{Gryzlov}.
 For every $x \in K$, let $\{V_n^x: n \in \w\}$ be a family of
open subsets of $X$ satisfying $\bigcap_{n<\w}V_n^x=\{x\}$.

Let $\mB = \{\bigcup_{i = 0}^k
V_{n_i}^{x_i} \supset K: x_0, ..., x_k \in K, n_0, ..., n_k \in
\w\}$. Note that $\bigcap \mB=K$ and $|\mB| \leq  2^{\w}$.
\end{proof}

Now, we have everything to prove  the announced result.

\begin{theorem} \label{main}
  Let $X$ be a  regular space with points $G_\delta$ such that
$\2$ has a winning strategy for the $\gfin^{\w_1}(\O, \O)$ game.
Then $|X| \leq 2^{\w}$.
\end{theorem}

\begin{proof}
 Let $\mu$ be a large enough regular cardinal and $M$ be an
elementary submodel of $H(\mu)$ such that $|M| = 2^\w$, $X,
\sigma \in M , \mathfrak c+1 \subset  M$ and $[M]^\w \subset M$.
Let $\mK = \{K
\subset X: K \in M$ and $K$ is good$\}$. It is enough to show
that $X = \bigcup \mK$, since each $K \in \mK$ is such that $|K|
\leq 2^\w$.

  Assuming the contrary, there is an $x \in X \setminus \bigcup
\mK$.
Let $K_0 = \bigcap_{\mC \in \O}\overline{\bigcup
\sigma(\mC)}$. Note that $K_0$ is definable in $M$ and so $K_0
\in \mK$. Working
inside of $M$ and using the three  previous lemmas, we obtain
that
there is an $\O' \in M$, such that $|\O'| \leq 2^\w$ and $K_0 =
\bigcap_{\mC \in \O'}\overline{\bigcup \sigma(\mC)}$. Since
$\mathfrak c+1\subset M$, we actually have  $\O' \subset M$ and
so
there is a $\mC_0 \in M \cap \O$ such that $x \notin \bigcup
\sigma(\mC_0)$. We now proceed by induction. Assume to have
already defined    open covers $\{\mC_\alpha :\alpha
<\xi\}\subset M$ and  define $s:\xi\to \O$ by letting $s(\alpha
)=\mC_\alpha $ for $\alpha <\xi$. Since $M$ is $\omega$-closed,
we actually have  $s \in M$. Therefore, $K_\xi  = \bigcap_{\mC
\in \O}\overline{\bigcup \sigma(s^\smallfrown \mC)}$ is
definable in $M$ and so it is again   an
element
of $\mK$. Then,  as before we can obtain a $\mC_\xi \in M \cap
\O$ such that $x \notin \bigcup \sigma(s^\smallfrown \mC_\xi)$.

But note that doing like this, we find a play of the game where
$\2$ loses although using a winning strategy.
\end{proof}

Note that we actually proved that under the hypothesis of Theorem
\ref{main}, $X = \bigcup_{\xi < \mathfrak c} K_\xi$, where each
$K_\xi$ is compact.  However, this is not enough to guarantee
that $\1$
wins in the long compact-open game without CH (see \cite{AB}).

Furthermore, note that with a simple modification in the previous
argument,
using a countable submodel we obtain the Telgarsky's result
(reproved by Scheepers in \cite{sch95}):

\begin{corollary}
  If $X$ is a  regular space where every compact set is a
$G_\delta$
and $\2$ has a winning strategy for the usual Menger  game
$\gfin(\O, \O)$,
then $X$ is $\sigma$-compact.
\end{corollary}

Since Theorem \ref{start} is actually true for $T_1$ spaces, we
could suppose that the same happens to Theorem \ref{main}.
But, Theorem \ref{main} drastically fails for $T_1$ spaces.
Indeed,
even under the stronger assumption that $\2$ has a winning
strategy in the ``short''  Menger game,  the cardinality
of a space with points $G_\delta$ can be very big.

\begin{example} \label{ex} If $\kappa$ is less than the first 
measurable cardinal, then there exists a $T_1$ space $X$ with
points
$G_\delta$ such that $\2$ has a winning strategy in
$\gfin(\O,\O)$ and $|X|\ge \kappa $. \end{example}
\begin{proof}  The example we need is just the space $X$
constructed by Juh\'asz in \cite{Juhasz} [Example 7.2].
Following the notation in \cite{Juhasz},  we have
$X=\bigcup\{X_n:n<\omega\}$, where $X_0=\kappa $. In
\cite{Juhasz} it is pointed out that  for a given $n<\omega$
every open  family  covering $X_{n+1}$ has a finite subfamily
covering all but
finitely many members of $X_n$. The latter assertion clearly
 implies that every open cover of $X$ has a finite subfamily
which covers
the whole $X_n$ and this in turn  guarantees an easy winning
strategy to $\2$ in $\gfin{(\O,\O)}$. \end{proof}

The original cardinality bound of Arhangel'ski\u\i\ as well as
most
of its variations    work for $T_2$ spaces. So, it is reasonable
to ask:
\begin{question} Does Theorem \ref{main} continue to hold for
$T_2$ spaces? \end{question}

 Recall that,  given a space  $X$, the symbol 
$X_\delta$ denotes  the
 space with the same underlying set $X$ with the topology 
generated by the
$G_\delta$ subsets of $X$. In \cite{AB} it  was shown
that Theorem \ref{start} is actually a consequence of the more
general
statement  that a winning strategy for $\2$ in
$\gone^{\omega_1}(\O,\O)$  implies that the Lindel\"of degree of
$X_\delta$ is at most $2^\omega$.  This seems to suggest a
possible further strengthening of Theorem \ref{main} as follows: 
if  $X$  is a  regular space where $\2$ has a
winning strategy  in $\gfin^{\omega_1}(\O,\O)$, then  
$L(X_\delta)\le 2^\omega$.
However, this conjecture drastically fails because  there are
compact $T_2$ spaces  such that the Lindel\"of degree of the
$G_\delta$-modification is much bigger than the continuum (see e.
g. \cite{ss} or \cite{usuba}), while  for every compact space 
$\2$ may win  in $\gfin^{\omega_1}(\O,\O)$ at the first inning!  

\smallskip
\section{Few remarks on the long tightness game}
We conclude the paper by looking at a long version of the
tightness game. One reason is in the similarity of
Theorem \ref{main} and Theorem \ref{main2}.

Given a space $X$ and a point $x\in X$, $\Omega_x$ denotes the
collection of all sets $A\subseteq X$ satisfying $x\in \overline
A$.  The tightness game
$\gone(\Omega_x, \Omega_x)$  is played between
players \1
and \2 in such a way that,  at every inning $n \in \omega $, \1
chooses a member $A_n \in \Omega_x$, and then  \2 chooses $a_n
\in
A_n$. \2 is declared the winner if, and only if, $\{a_n:
n \in \omega \} \in \Omega_x$ (see \cite{ABD} for much more).

 If the previous game consists of $\omega_1$-many innings, then
we have the long tightness game $\gone^{\w_1}(\Omega_x,
\Omega_x)$.

\begin{theorem} \label{main2}
  If $X$ is a regular space that has a dense subset $E$ with $|E|
\leq 2^\w$ and $\2$ has a winning strategy  in the
game $\gone^{\w_1}(\Omega_p, \Omega_p)$ for some $p \in X$, then
$\chi(p,X) \leq 2^\w$.
\end{theorem}

\begin{proof}
  Let $\sigma$ be a winning strategy for $\2$.
Let $\mu$ be a large enough regular cardinal and $M$ be an
elementary submodel of $H(\mu)$ such that
$ E \subset M$, $X,  \sigma \in M$,
$[M]^\w \subset M$ and $|M| = 2^\w$. For every sequence $s \in
\Omega_p^{<\w_1}$, there is a
neighbourhood $V_s$ of $p$  such that for every $x \in V_s$,
there is
a $D \in
\Omega_p$ such that $x = \sigma(s^\smallfrown D)$. We will call
such a neighbourhood \emph{good}. To verify the existence of
$V_s$, assume the contrary and let $D$ be the set of all  $x\in
X$ such that $\sigma(s^\smallfrown A)\ne x$ for each $A\in
\Omega_p$. But then $D\in \Omega_p$ and so $\sigma(s^\smallfrown
D)\in D$, in contrast with the definition of $D$.

Now, to prove the theorem  it
is enough to show that $\mV = \{V \subset X: V \in M$ and $V$ is
good$\}$ is a local base  at $p$. Assume the contrary.  Then, by
regularity, there is an open neighborhood $W$ of $p$ such that $V
\not \subset \overline W$ for every $V \in \mV$. Let $V_0$
be an open set  such that for every $x \in V_0$ there is a $D \in
\Omega_p$ such
that $x = \sigma(D)$. $V_0$ is definable in $M$ and so  $V_0 \in
\mV$.  Besides, by density,
there is an $e_0 \in (V_0  \setminus \overline W) \cap E$. Note
that
$e_0$ is in $M$, therefore there is a $D_0$ such that
$\sigma(D_0) = e_0$. Now,   we proceed by induction, by assuming
to have already defined  points $e_\alpha \in E$ and sets
$D_\alpha \in \Omega_p$ for $\alpha <\xi$.  let $s=\{(e_\alpha ,
D_\alpha ):\alpha <\xi\}\in( \Omega_p \cap M)^{<\w_1}$.  Since
$M$ is
countably closed, $s \in M$. Therefore, there is an open
neighborhood $V_s$ of $p$ such that for every $x \in V_s$,
$\sigma(s^\smallfrown D) = x$. Again, $V_s \in M$. As before, we
can take $e_\xi \in (V_s \setminus \overline W) \cap E$ and then
choose $D_\xi$ such that $e_\xi = \sigma(s^\smallfrown D_\xi)$.
Note that $D_\xi \in M$.  But, playing like this, at the end
$\2$ would
loose the game -- a contradiction.
\end{proof}

One may wonder if the above theorem is the best possible, namely
if we could get $\chi(p,X)\le \w_1$.    This obviously  happens
by assuming $2^\w=\w_1$, but   the next example  show it is no
longer true
without the Continuum Hypothesis.

\begin{example} A regular space $X$ with a dense set
$E$  of size $2^\omega$ and a point $p$ such that $\2$ has a
winning strategy in $\gone^{\w_1}(\Omega_p, \Omega_p)$ and
$\chi(p,X)=2^\w$. \end{example}
\begin{proof}  Let $E$ be a set of cardinality $2^\w$ with the
discrete topology and let $X=E\cup\{p\}$ be the one-point
Lindel\"ofication of $E$. Observe that  $U$ is a neighbourhood of
$p$ in $X$ if and only if $p\in U$ and $|X\setminus U|\le
\omega$.  We have $\chi(p,X)=2^\w$.
 Indeed, if $\mathcal U$ is a collection of neighborhoods of $p$
satisfying $|\mU|<2^\w$, then  $|E\setminus \bigcap \mU|\le
|\mU|\omega<2^\w$ and so $|\bigcap \mU|=2^\w$, which in turn
implies that $\mU$ cannot be a local base. On the other hand,
\2 has an
easy winning strategy in $\gone^{\w_1}(\Omega_p,\Omega_p)$: fix
$\xi<\w_1$ and suppose that  $e_\alpha
$ is the point \2 has chosen at the inning $\alpha <\xi$. If   at
the $\xi$-inning \1 plays $A_\xi\in \Omega_p$, then \2  simply
takes a point $e_\xi\in A_\xi\setminus \{e_\alpha :\alpha
<\xi\}$.  This can be done because $A_\xi$ is uncountable. Now,
at the end of the game \2 has chosen an uncountable set of
points  and so  he wins. \end{proof}

Let us denote by $\D$ the collection of all dense subsets of a
given topological space. Note that if $\2$ has a winning strategy
for the game $\gone^{\w_1}(\D, \D)$, then the density of the
space is less or equal to $\w_1$. Therefore, the next result can
be proved with almost the same argument that in Theorem
\ref{main2}:

\begin{theorem}
  If $X$ is a regular space where $\2$ has a winning strategy  in

the game
$\gone^{\w_1}(\D, \D)$ then $\pi w(X) \leq 2^\w$.
\end{theorem}

Comparing Theorems \ref{start} and \ref{main},
one may  be tempted to conjecture that a result similar to
Theorem
\ref{main2} continues to hold
for $\gfin$ instead of $\gone$.
 But,  it turns out that even the difference  between $\gn{2}$
and
$\gone$ can be very big - here $\gn{2}$ is the game where
$\2$ is allowed to take at most  2 points instead of just one.
Indeed, even  the fact
that  \2 always wins the ``short'' game
$\gn{2}(\Omega_p,\Omega_p)$ does not guarantee that \2 has a
winning strategy in the long tightness game, as the following
example from \cite{ABD} shows:

\begin{example}\label{long}
  A zero-dimensional $T_1$ space where $\2$ has a winning
strategy in $\gn{2}(\Omega_p, \Omega_p)$  and $\1$ has a
winning strategy in $\gone^{\w_1}(\Omega_p, \Omega_p)$.
\end{example}

\begin{proof}
  Let $X = \{p\} \cup \w^{<\w_1}$ with the followin topology:
every point other than $p$ is isolated. The basic neighborhoods
at $p$ are of the form
  \[\{p\} \cup \w^{<\w_1} \setminus F\]
  where $F$ is the union of finitely many branches in
the tree $\w^{<\w_1}$. Let us show that $\1$ has a winning
strategy in
$\gone^{\w_1}(\Omega_p, \Omega_p)$. $\1$ starts with $D_0 =
\{\seq{n}: n \in \w\}$. Let $s$ be the choice of $\2$. Note that
then $\1$ can play $D_1 = \{s^\smallfrown n: n \in \w\}$.
Indeed, By playing in this way,  at a certain inning the set of
all choices of $\2$ is a function $s:\alpha +1\to \w$.  Then
$\1$ simply can play $D = \{s^\smallfrown n: n \in \w\}$. Note
that playing like this, at the end all of the choices of $\2$
forms a branch thus $\1$ wins.

Now let us see that $\2$ has a winning strategy for the
$\gn{2}(\Omega_p, \Omega_p)$ game. It is enough to show
that, for each  $n \in \w$,
the set  of all the answers played by \2 in the first $n$ innings
includes a set $\{s_1, ..., s_n\}$   with the property that
no branch contains two elements of it.

Let us proceed by induction.  If, in the first inning,   $\1$
plays $A_1$, then
$\2$
chooses $\{s_1, s_2\} \subset A_1$ such that $s_1$ and $s_2$ are
not in
the same branch. Suppose that at the end of the $n$-th inning,
the set
of all answers of  $\2$ contains a set  $\{s_1, ..., s_n\}$
satisfying our assumption.  Let $A_{n + 1}$ be the play of  $\1$
 at the inning $n+1$.  If there is a point in $A_{n + 1}$ that
lies in a
branch missing $\{s_1,..., s_n\}$, then  $\2$ chooses this point
together with some other one. In  the remaining case, since $p$
is
in the closure of $A_{n+1}$,  there
is at least one $s_i$ and two incompatible  elements  $a_1, a_2
\in A_{n +1}$ such that   $s_i \subset a_1$ and $s_i\subset
a_2$.   The
answer of  $\2$ in the $(n+1)$-th inning will be  just
$\{a_1,a_2\}$.
Observe that  every branch meets the set $\{s_j: j \neq i\} \cup
\{a_1, a_2\}$ in at most one point.
\end{proof}

In the previous proof, we did not use that much information
about the height of the tree. Therefore, we can easily modify the
example to obtain the following:

\begin{proposition}
  There is a zero-dimensional $T_1$  space $X$ and a point $p\in
X$   such that $\2$ has a
winning
strategy in $\gn{2}(\Omega_p, \Omega_p)$, $|X|
= 2^\w$ and $\chi(p,X) > 2^\w$.
\end{proposition}

In particular, this shows that we cannot generalize Theorem
\ref{main2}  for the version where $\2$ is allowed to pick two
points instead of one!

Finally, a simplified version of the above construction gives:

\begin{proposition}
  There is a countable zero-dimensional space$X$  where $\2$ has
a winning strategy
in $\gn{2}(\Omega_p, \Omega_p)$ but $\chi(p,X) > \w$.
\end{proposition}

Inspired by the Example \ref{long}, we finish with a similar
construction that may serve as an example of the ideas used
here. Let $T$ be an uncountable tree
with no uncountable chains (\emph{e. g.} an $\omega_1$-Aronszajn
three)  and consider $X = T \cup \{p\}$ with
the following topology: every point of $T$ is isolated and the
neighborhoods of $p$ are of the form $X \smallsetminus \bigcup F$
where $F$ is a finite collections of branches of $T$. Note that
$\1$ cannot repeat the analogous strategy made in the Example
\ref{long}, since that would imply the existence of an
uncountable branch, which is impossible. Moreover, it is very
easy  for $\2$ to guarantee his own victory. Indeed, it is enough
to him to play in a manner where he ends up by playing
uncountably
many distinct  points. This is enough since in a tree any
uncountable set
contains either an uncountable branch or an infinite antichain.

\smallskip

\section{Acknowledgements}
The authors thank Santi Spadaro for calling their attention
to Example \ref{ex} and Toshimichi Usuba for  his valuable
comments.

The research that led to the present paper was done during the
visit of the first-named author at the University of Catania and
it was partially
supported by
a grant of the
group GNSAGA of INdAM. The first author is also supported by
FAPESP, grant 2017/09252-3.

\smallskip

\bibliographystyle{abbrv}

\def\cprime{$'$}

\end{document}